\documentclass[proc]{edpsmath}
\usepackage{bm}
\usepackage{tikz}
\usetikzlibrary{quantikz}
\usetikzlibrary{positioning}
\usepackage{multirow} 
\usepackage{booktabs}
\usepackage{algorithm}
\usepackage{algpseudocode}
\usepackage{float}
\usepackage{hyperref}
\usepackage{amsthm}

\usepackage[bf,textfont={it}]{caption}

\newtheorem{theorem}{Theorem}

\newtheorem{proposition}{Proposition}

\newcommand{\I}{\mathbb{I}}
\newcommand{\X}{\mathtt{X}} 
\newcommand{\Y}{\mathtt{Y}} 
\newcommand{\Z}{\mathtt{Z}}

\newtheorem{Proposition}{Proposition}
\newtheorem{Lemma}{Lemma}

\begin{document}

%%-----------------------------
%%      the top matter
%%-----------------------------
\title{Mitigating Barren Plateaus via Domain Decomposition in Variational Quantum Algorithms for Nonlinear PDEs}

\author{Laila S. Busaleh}
\address{Department of Mathematics, College of Science and Humanities, 
Imam Abdulrahman Bin Faisal University, Jubail, Saudi Arabia.\\
Applied Mathematics and Computational Science, 
CEMSE Division, King Abdullah University of Science and Technology (KAUST), 
Thuwal 23955-6900, Kingdom of Saudi Arabia (lsbusaleh@iau.edu.sa).}

\author{Jeonghyeuk Kwon}
\address{Kyung Hee University, Department of Mathematics, Seoul 02447, Republic of Korea (jeonghyeuk@khu.ac.kr)}
\author{Orlane Zang}
\address{Orange Innovation, Université de Picardie Jules Verne, France.}

\author{Muhammad Hassan}
\address{Technische Universität München, Department of Mathematics, Boltzmannstrasse 3, Garching 85748, Germany}
\author{Yvon Maday}
\address{Sorbonne Université, CNRS, Université Paris Cité, Laboratoire Jacques-Louis Lions (LJLL),
F-75005 Paris, France.}
%
%\dedicated{\it Dedicated to Maurice Dupont} %if necessary
%

\begin{abstract}
Barren plateaus present a major challenge in the training of variational quantum algorithms (VQAs), particularly for large-scale discretizations of nonlinear partial differential equations. In this work, we introduce a domain decomposition framework to mitigate barren plateaus by localizing the cost functional. Our strategy is based on partitioning the spatial domain into overlapping subdomains, each associated with a localized parameterized quantum circuit and measurement operator. Numerical results for the time-independent Gross–Pitaevskii equation show that the domain-decomposed formulation, allowing subdomain iterations to be interleaved with optimization iterations, exhibits improved solution accuracy and stable optimization compared to the global VQA formulation.
\end{abstract}

\begin{resume}
Les \textit{plateaux stériles} (\textit{barren plateaus}) représentent une difficulté majeure pour l’entraînement des algorithmes quantiques variationnels, notamment lorsqu’ils sont appliqués à des discrétisations à grande échelle d’équations aux dérivées partielles non linéaires. Dans ce travail, nous introduisons une approche fondée sur  une méthode de décomposition de domaine, dont l’objectif est d’atténuer ce phénomène par une localisation de la fonction de coût. Notre approche repose sur une partition du domaine spatial en sous-domaines avec recouvrement, chacun étant associé à un circuit quantique paramétré local ainsi qu’à un opérateur de mesure dédié.
Les résultats numériques obtenus pour l’\textit{équation de Gross–Pitaevskii stationnaire} montrent que cette formulation basée sur la décomposition de domaine, permettant d'enchevétrer les itérations par sous domaine avec les itérations d'optimisation, conduit à une meilleure précision de la solution et à une optimisation plus stable que la formulation globale des algorithmes quantiques variationnels.
\end{resume}
\maketitle
%%-----------------------------
%%      your text
%%-----------------------------

%\laila{We have to do the following:
%\begin{itemize}

% \item Change the algorithm. (Done)
% \item Please note that the maximum number of pages should be $20$. {\color{blue} Otherwise, it is possible to exceed 20 pages by specifying in the email the reason why there are more than 20 pages}.
% \item I reordered and merged some subsections in Sections 4 and 5 to improve the flow. 
% Don't worry — I kept the original ordering and organization, so we can easily revert if you don't like the new structure :)
% \end{itemize}}
% {\color{magenta} Jeong: Is Figure 7 in the right place?} \laila{I moved it from subsection 5.2 to the right subsection (5.3) now its figure 9  
% \eqref{fig:dd_vs_full}}

%\YM{In table 2, the Newton solution is the same with $N= 2^{10}$ or with $N=2^n$ ??} {\color{magenta} Jeong: The Newton reference solution is computed on a fixed high-resolution grid with $N=2^{11}$ points and used as the reference for all experiments.}

%\vfill\eject

\section*{Introduction}

%The numerical solution of nonlinear PDEs remains a central challenge in mathematics and computational physics. In his seminal lecture, Feynman suggested that quantum systems could be simulated more efficiently on quantum devices than on classical computers \cite{Feynman1982}, laying the foundation for quantum approaches to scientific computing. Since then, extensive progress has been made in the development of numerical methods on classical computers. Nevertheless, as the complexity of the problem increases, the computational demands in terms of time and memory grow rapidly, often making classical approaches challenging for large-scale or high-dimensional applications.

%Quantum computing has emerged as an alternative paradigm for overcoming these limitations, with the potential to provide exponential speedups for certain problems \cite{Arute2019, Abhijith2018}. However, current devices operate in the noisy intermediate-scale quantum (NISQ) regime \cite{Preskill2018}, characterized by limited qubit numbers, short coherence times, and imperfect gate operations. In this regime, fault-tolerant quantum algorithms involving deep quantum circuits are not yet feasible, and direct quantum implementations of numerical solvers may suffer from loss of accuracy and instability due to hardware noise.

The numerical solution of nonlinear partial differential equations (PDEs) remains a central challenge in mathematics and computational physics. Although substantial progress has been made in the development of classical numerical methods, the computational cost in both time and memory still grows rapidly with problem size and dimensionality, often making large-scale or high-dimensional applications prohibitively expensive.

In this context, quantum computing has emerged as a promising alternative paradigm for scientific computing. The idea can be traced back to Feynman’s seminal lecture, in which he suggested that quantum systems could be simulated more efficiently on quantum devices than on classical computers \cite{Feynman1982}. Building on this perspective, quantum algorithms have since been developed with the potential to provide significant, and in some cases exponential, speedups for specific computational tasks \cite{Arute2019, Abhijith2018}.

However, current quantum hardware operates in the noisy intermediate-scale quantum regime \cite{Preskill2018}, characterized by limited qubit counts, short coherence times, and imperfect gate operations. As a result, fault-tolerant algorithms requiring deep quantum circuits remain out of reach, and direct quantum implementations of numerical solvers may still suffer from reduced accuracy and instability due to hardware noise.

These hardware constraints have led to the development of hybrid quantum–classical algorithms, where a parameterized quantum circuit is trained within a classical optimization loop. Such methods have gained popularity in quantum simulation, optimization, and machine learning. Owing to their simplicity and hardware efficiency, random circuit architectures are often employed as initial ansatz for exploring the quantum state space \cite{McClean2018}. Within this hybrid paradigm, variational quantum algorithms (VQAs) provide a flexible framework that combines the expressive power of quantum circuits with the optimization capabilities of classical routines.

Many nonlinear PDEs admit a variational structure in which solutions arise as critical points of an energy functional. This structure is particularly natural for VQAs, which approximate ground states by minimizing expectation values of parameterized quantum states. Variational quantum approaches for differential equations have recently been proposed in several contexts, including for nonlinear Schrödinger equations and reaction–diffusion systems \cite{Albino2022, GarciaMolina2022, Liu2023, Lubasch2020, Pool2024}. Recent advances in variational quantum optimization also suggest that, on NISQ devices, local or greedy optimization procedures may in some settings be preferable to fully global updates \cite{feniou2025greedy}. The variational structure of these equations therefore makes them particularly attractive for hybrid quantum optimization.

In a VQA, a parametrized quantum circuit encodes candidate solutions, while a classical optimizer iteratively updates parameters by minimizing a cost function that reflects the PDE residual or energy functional. A key obstacle, however, is the barren plateau (BP) phenomenon, where gradients vanish exponentially with the system size, rendering optimization intractable \cite{McClean2018}. BP may arise from several sources, including deep or highly expressive ansatz circuits, concentration of measure effects, hardware noise, and the use of global cost functions \cite{McClean2018, Cerezo2021, Ragone2024}. In particular, the choice of observable plays a critical role: global observables typically induce BP, whereas localized observables can preserve gradient magnitudes and improve trainability \cite{Cerezo2021, Ragone2024}. A comprehensive overview of these mechanisms and mitigation strategies is provided in \cite{Larocca2024ReviewBP}.

In this work, we propose a domain decomposition strategy to mitigate BP in variational quantum solvers by partitioning the computational domain and localizing the cost function, thereby ensuring that the resulting optimization landscape retains sufficient structure to remain trainable. 
This formulation enables the interleaving of subdomain iterations and optimisation steps, leading to reinitialise the local quantum training procedure before BP effects can significantly hinder the minimisation process.
We demonstrate the effectiveness of this approach on the time-independent Gross–Pitaevskii equation (GPE), a nonlinear Schrödinger equation introduced in the context of weakly interacting Bose gases \cite{LandauLifshitz1977}. The remainder of the paper is organized as follows.
Section~\ref{sec:problem_formulation} introduces the GPE, its variational formulation, and the spectral discretization.
Section~\ref{sec:vqa_formulation} presents the variational quantum formulation.
Section~\ref{sec:bp} analyzes the behavior of the full-domain VQA and its scalability.
Section~\ref{sec:dd_strategy} introduces the proposed domain decomposition framework.
Section~\ref{sec:numerical_results} reports numerical experiments comparing the full-domain and domain-decomposed formulations. Finally, the appendix presents several results clarifying the relation between barren plateau effects and the dimension of the Lie algebra generated by the variational circuit, both in the global setting and in the domain-decomposition framework.

\section{Problem Formulation and Discretization Strategy}
\label{sec:problem_formulation}
We consider the time-independent, one-dimensional, periodic GPE with an external potential and cubic nonlinearity:
\begin{equation}
    -\frac{1}{2} \psi''
    + V\,\psi
    + \kappa\,|\psi|^2 \psi
    = \lambda\,\psi, 
    \quad \text{in the weak sense on } H^1(\Omega),
    \label{eq:GPE}
\end{equation}
subject to the normalization constraint
\begin{equation}
    \int_0^{2\pi} |\psi(x)|^2 \, dx = 1.
\end{equation}
Here, $\Omega =\mathbb{R}/2\pi\mathbb{Z}$ is a one-dimensional torus, isomorphic to the unit cell $[0, 2\pi)$ of the lattice $2 \pi\mathbb{Z}$, $V \in L^\infty(\Omega)$ is the external $2\pi$-periodic potential, $\kappa\in\mathbb{R}$ controls the nonlinear interaction strength, and $\lambda$ is the associated chemical potential.  It is well-known that Equation~\eqref{eq:GPE} is the Euler--Lagrange equation of the following constrained energy minimisation problem
\begin{equation}\label{eq:GPE_Variational_1}
    \min_{\psi\in H^1(\Omega),\, \|\psi\|_{L_2}=1}
    E(\psi),
\end{equation}
where $E\colon H^1(\Omega) \rightarrow \mathbb{R}$ is the Gross-Pitaevskii energy functional given by
\begin{equation}\label{eq:GPE_Variational_2}
    E(\psi)
    = \frac12 \int_0^{2\pi} |\psi'(x)|^2 \, dx
    + \int_0^{2\pi} V(x) |\psi(x)|^2 \, dx
    + \frac{\kappa}{2} \int_0^{2\pi} |\psi(x)|^4 \, dx.
\end{equation}
The first, second, and third terms in the energy functional $E$ correspond to kinetic, potential, and interaction energies, respectively. In the sequel we will adopts the variational formulation \eqref{eq:GPE_Variational_1}-\eqref{eq:GPE_Variational_2} of the time-independent GPE because our objective is to approximate ground states via variational quantum optimization. 
%This, of course, involves training a parameterized quantum state by minimizing an energy expectation value. 
Note that the variational structure above is particularly well suited for VQAs, since the wavefunction $\psi$ can be interpreted as a normalized quantum state-vector and the energy $E(\psi)$ corresponds to the expectation value of the Hamiltonian.

%The functional $E$ therefore serves directly as the objective function in the hybrid quantum–classical framework.

%that the ground state (i.e., lowest eigenfunction) of this equation can be characterized as a constrained minimizer of the associated Gross-Pitaevskii energy functional. 

%This perspectiveWe adopt this variational formulation because our objective is to approximate ground states via variational quantum optimization In the variational quantum setting, a parameterized quantum state is trained by minimizing an energy expectation value. 

\subsection{Spectral Discretization}
We consider a uniform discretization of the periodic interval $\Omega=[0,2\pi)$. Since the discretized solution is intended to be represented on a quantum computer, we choose the number of grid points as $N=2^n$, so that the grid can be naturally encoded on $n$ qubits. The grid points are then given by
%We consider a uniform discretization of the periodic interval $\Omega=[0,2\pi)$ with $N=2^n$ grid points,
\[
x_j = j\,\Delta x,\qquad j=0,1,\dots,N-1,\qquad \Delta x = \frac{2\pi}{N}.
\]
The discrete wavefunction is represented by the vector of grid point-values
$\bm{\psi}=(\psi_j)_{j=0}^{N-1}\in\mathbb{C}^N$ and satisfies the discrete normalization
\begin{equation}\label{eq:disc_norm}
\Delta x \sum_{j=0}^{N-1} |\psi_j|^2 = 1.
\end{equation}
%In the Fourier pseudo-spectral setting, we identify $\bm{\psi}$ with the unique
%$2\pi$-periodic trigonometric polynomial (Fourier interpolant) $\psi_N\in V_N$ such that
%$\psi_N(x_j)=\psi_j$ for all $j$, where
%\[
%V_N := \mathrm{span}\{e^{i\xi_\ell x}:\ \ell\in\mathcal{I}_N\},\qquad
%\mathcal{I}_N=\left\{-\frac{N}{2},\dots,\frac{N}{2}-1\right\},\qquad
%\xi_\ell=\ell.\]
In the Fourier pseudo-spectral setting, we identify $\bm{\psi}$ with the
$2\pi$-periodic trigonometric polynomial (Fourier interpolant)
$\psi_N\in V_N$ such that
\[
\psi_N(x_j)=\psi_j \qquad \text{for all } j,
\]
where
\[
V_N := \mathrm{span}\Bigl(\{e^{i\ell x}\}_{\ell\in\mathcal{I}'_N}\cup\{ \cos(Nx/2)\}\Bigr),
\qquad
\mathcal{I}'_N=\left\{-\frac{N}{2}+1,\dots,\frac{N}{2}-1\right\}, \quad \mathcal{I}_N = \mathcal{I}'_N\cup 
\left\{\frac{N}{2}\right\}.
\]
where the additional mode 
$\cos(Nx/2)$ accounts for the Nyquist frequency. Equivalently, letting $\widehat{\psi}_\ell$ be the DFT coefficients of the grid values $\{\psi_j\}_{j=0}^{N-1}$,
the interpolant admits the Fourier series representation
\begin{equation}\label{eq:interp_series}
\psi_N(x)= \sum_{\ell\in\mathcal{I}'_N}\widehat{\psi}_\ell\,e^{i\ell x} + \sqrt{2} \widehat{\psi}_{\frac{N}{2}} \cos(Nx/2).%\psi_N(x)=\frac{1}{\sqrt{N}}\sum_{\ell\in\mathcal{I}_N}\widehat{\psi}_\ell\,e^{i\ell x}.
\end{equation}
%$$\psi_N(x)= \sum_{\ell\in\mathcal{I}'_N}\widehat{\psi}_\ell\,e^{i\ell x} + \widehat{\psi}_{\frac{N}{2}} \cos(Nx/2).\hskip 7truecm (6')$$

Under this representation, the second derivative operator $\partial_{xx}$ is diagonal in Fourier space with eigenvalues $-\ell^2$.
Consequently, for the interpolant $\psi_N$, the kinetic energy satisfies
\begin{equation}\label{eq:kinetic_identity}
\frac12 \int_0^{2\pi} |\psi_N'(x)|^2 \, dx
=
\pi
\sum_{\ell\in\mathcal{I}_N}
|\ell|^2
|\widehat{\psi}_\ell|^2.
\end{equation}
In implementations based on real FFTs one may equivalently use a sine/cosine representation; for conciseness we use the complex exponential notation in \eqref{eq:interp_series}. The potential and interaction terms are evaluated by pointwise quadrature on the uniform grid. The resulting discrete energy functional is
\begin{equation}
E_n(\bm{\psi})
=
\pi 
\sum_{\ell} |\ell|^2\,|\widehat{\psi}_\ell|^2
+
\Delta x \sum_{j=0}^{N-1} V(x_j)\,|\psi_j|^2
+
\frac{\kappa\,\Delta x}{2}
\sum_{j=0}^{N-1} |\psi_j|^4.
\label{eq:EXN}
\end{equation}
The finite-dimensional constrained minimization problem becomes
\begin{equation}
\min_{\bm{\psi}\in\mathbb{C}^N}
E_n(\bm{\psi})
\quad
\text{subject to}
\quad
\Delta x \sum_{j=0}^{N-1} |\psi_j|^2 = 1.
\end{equation}

% Introducing a discrete Lagrange multiplier $\lambda_h$, the first-order optimality condition yields the nonlinear eigenvalue problem
% \begin{equation}
% -\frac12 D_2 \bm{\psi}
% + V \odot \bm{\psi}
% + \kappa |\bm{\psi}|^2 \odot \bm{\psi}
% =
% \lambda_h \bm{\psi},
% \end{equation}
% where $D_2$ denotes the spectral discretization of the second derivative operator and $\odot$ indicates pointwise multiplication \YM{the interpretation of $D_2$ is not clear since $\psi$ is not well defined}.

\section{Variational Quantum Formulation}
\label{sec:vqa_formulation}

In this section, we describe the variational quantum formulation
used to approximate the ground state of the discrete energy functional
$E_n$ defined in \eqref{eq:EXN}.
Roughly speaking, the discrete wavefunction $\bm{\psi} \in \mathbb{C}^N$
is encoded as the state-vector generated by a parameterized quantum circuit,
and the energy is minimized with respect to the circuit parameters. Let the spatial domain be discretized into $N$ grid points.
In the variational setting, the discrete wavefunction depends on the circuit parameters and is written as
\[
\bm{\psi}(\bm{\theta})
=
(\psi_0(\bm{\theta}), \dots, \psi_{N-1}(\bm{\theta}))^\top
\in \mathbb{C}^N \qquad \text{for some parameters}\qquad \bm{\theta} \in [0, 2\pi)^m.
\]
The normalization constraint
\[
\Delta x \sum_{j=0}^{N-1}
|\psi_j(\bm{\theta})|^2
= 1
\]
is equivalent to the Euclidean normalization of the rescaled vector
\[
\bm{\phi}(\bm{\theta})
:=
\sqrt{\Delta x}\,\bm{\psi}(\bm{\theta}),
\qquad
\|\bm{\phi}(\bm{\theta})\|_2 = 1.
\]
Thus, $\bm{\phi}(\bm{\theta})$ may be regarded as a quantum state-vector
in a Hilbert space of dimension $N = 2^n$, encoded on $n$ qubits. Under this identification, the discrete energy functional is evaluated at the parameterized wavefunction $\psi(\theta)$ : 
$E_n(\psi(\theta))$,
where $E_n$ is defined in \eqref{eq:EXN}.
In the variational framework, the state-vector is restricted to the manifold
generated by a parameterized unitary circuit,
\[
|\bm{\phi}(\bm{\theta})\rangle
=
U(\bm{\theta})|0\rangle,
\qquad
\bm{\theta} \in [0, 2\pi)^m .
\]
The ground-state computation is therefore reformulated
as the parameter optimization problem
\begin{equation}
\label{eq:global-vqa}
\bm{\theta}^\star
=
\underset{\bm{\theta}\in [0, 2\pi)^m}{\text{\rm argmin}}
\mathcal{C}(\bm{\theta}),
\qquad \text{with} \quad
\mathcal{C}(\bm{\theta}) = {\mathfrak E}_n(\bm{\phi}(\bm{\theta}))
\equiv
E_n\big(\bm{\psi}(\bm{\theta})\big)   .
\end{equation}
To parameterize the state manifold, we adopt the hardware-efficient ansatz introduced in \cite{Kocher2025} and illustrated in Figure~\ref{fig:ansatz_n5},
which was successfully applied to the time dependant GPE in the context of Bose-
Einstein condensation within a hybrid pseudospectral–variational quantum framework.
The circuit consists of  $d=d(n)$ entangling layers, where the circuit depth $d$ is chosen as a function of the number  $n$ of qubits involved in the circuit.
Each layer applies single-qubit rotations
\[
R_x(\theta_x^{(\ell,q)}),
\quad
R_z(\theta_z^{(\ell,q)})
\]
to every qubit $q = 0, \dots, n-1$,
followed (for $\ell < d$) by a ring of CNOT gates connecting nearest neighbors together with an additional CNOT between the last and first qubits.
The final layer consists only of rotations.
The total number of parameters is thus
\[
m = 2n(d+1). 
\]

\begin{figure}[h!]
\centering
\begin{quantikz}[row sep=0.4cm, column sep=0.4cm]
\qw & \gate{R_x(\theta_{x}^{(0,0)})} & \gate{R_z(\theta_{z}^{(0,0)})} & \ctrl{1} \gategroup[wires=5,steps=7,style={solid,draw=black,inner sep=6pt},label style={yshift=0.2cm}]{$\times$ d} & \qw & \qw & \qw & \targ{} & \gate{R_x(\theta_{x}^{(\ell,0)})} & \gate{R_z(\theta_{z}^{(\ell,0)})} & \qw \\
\qw & \gate{R_x(\theta_{x}^{(0,1)})} & \gate{R_z(\theta_{z}^{(0,1)})} & \targ{} & \ctrl{1} & \qw & \qw & \qw & \gate{R_x(\theta_{x}^{(\ell,1)})} & \gate{R_z(\theta_{z}^{(\ell,1)})} & \qw \\
\qw & \gate{R_x(\theta_{x}^{(0,2)})} & \gate{R_z(\theta_{z}^{(0,2)})} & \qw & \targ{} & \ctrl{1} & \qw & \qw & \gate{R_x(\theta_{x}^{(\ell,2)})} & \gate{R_z(\theta_{z}^{(\ell,2)})} & \qw \\
\qw & \gate{R_x(\theta_{x}^{(0,3)})} & \gate{R_z(\theta_{z}^{(0,3)})} & \qw & \qw & \targ{} & \ctrl{1} & \qw & \gate{R_x(\theta_{x}^{(\ell,3)})} & \gate{R_z(\theta_{z}^{(\ell,3)})} & \qw \\
\qw & \gate{R_x(\theta_{x}^{(0,4)})} & \gate{R_z(\theta_{z}^{(0,4)})} & \qw & \qw & \qw & \targ{} & \ctrl{-4} & \gate{R_x(\theta_{x}^{(\ell,4)})} & \gate{R_z(\theta_{z}^{(\ell,4)})} & \qw
\end{quantikz}
\caption{Hardware-efficient ansatz for $5$ qubits.
Each entangling layer consists of single-qubit rotations
followed by a ring of CNOT gates.}
\label{fig:ansatz_n5}
\end{figure}
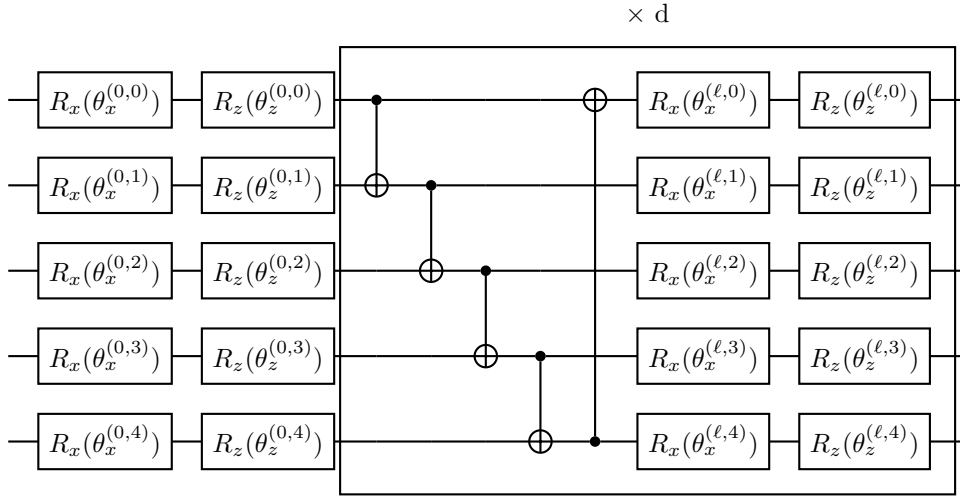
In the present work, all circuits are simulated at the state-vector level.
Gradients with respect to $\bm{\theta}$ are computed exactly within the classical simulator,
and a quasi-Newton method (BFGS) is employed to minimize $\mathcal{C}(\bm{\theta})$.
The resulting optimization problem is nonlinear and nonconvex,
and its performance depends critically on the structure of the landscape $\mathcal{C}(\bm{\theta})$. We therefore begin by examining the optimization dynamics of the full-domain formulation.

\section{Full-Domain Behavior and Barren Plateaus}
\label{sec:bp}

We analyze the global (full-domain) VQA formulation \eqref{eq:global-vqa}
under the periodic potential
\begin{equation}
\label{potential}
    V(x)=1-\cos(x),
\end{equation}
on the periodic domain $\Omega=[0,2\pi)$ and study its scalability as the system size $N=2^n$ increases. 
We conduct numerical experiments using $n=7,8,$ and $9$ qubits with circuit depths 
$d=100,200,$ and $400$, respectively\footnote{In the absence of a universal prescription, we take $d=d(n)$ to scale moderately with $n$.}. 
Figure~\ref{fig:full_domain_errors} shows the optimization dynamics of the full-domain formulation as a function of the number of training steps.

\begin{figure}[h!]
\centering
\includegraphics[width=\textwidth, trim={0cm 9cm 0cm 0cm}, clip]
{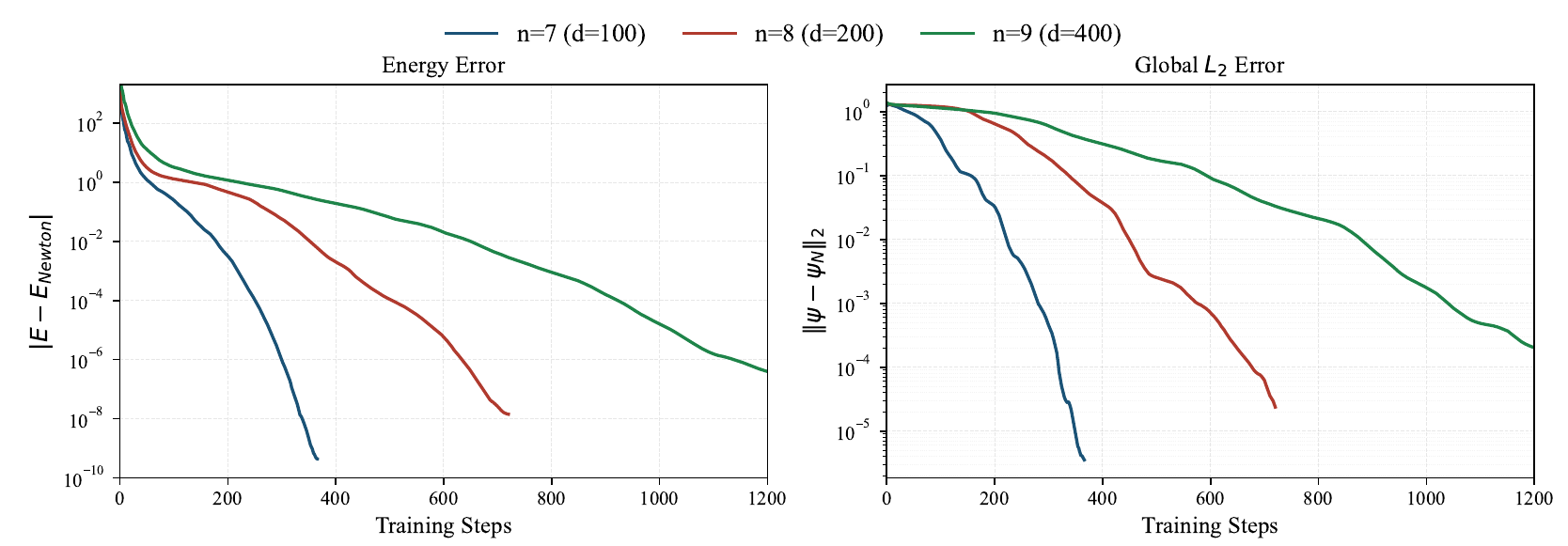}
\includegraphics[width=0.53\textwidth, trim={0cm 0cm 14cm 9.7cm}, clip]
{figures/full_domain_summary_final_2x2.pdf}
\caption{
Full-domain training dynamics for $n=7,8,9$. Top left: energy error $|E-E_{\rm Newton}|$ relative to the Newton reference. Top-right: wavefunction error $\|\psi-\psi_N\|_2$. Bottom: Relative energy change between consecutive iterations for the full-domain formulation.
The BFGS stopping tolerance is set to $10^{-20}$.
}
\label{fig:full_domain_errors}
\end{figure}

% \begin{figure}[h!]
% \centering
% \includegraphics[width=0.5\textwidth, trim={0cm 0cm 14cm 9.7cm}, clip]
% {macro-latex-proc/figures/full_domain_summary_final(2x2).pdf}
% \caption{
% }
% \label{fig:rel_energy_change}
% \end{figure}

A qualitative change is observed in the optimization dynamics: as $n$ increases, the decrease in the energy error becomes progressively slower and the energy curve exhibits a significantly flatter profile. This indicates an increasing difficulty in identifying effective descent directions, a behavior reminiscent of BP phenomena observed in large-scale quantum circuits. 

We believe that the ``relative energy change'' curves  are particularly informative in this regard, as they exhibit three distinct regimes: an initial decay phase (for training steps up to approximately 100), followed by a common plateau region (at the level of $10^{-2}$) whose extent increases with $n$ (roughly $[100,150]$ for $n=7$, $[100,250]$ for $n=8$, and $[100,600]$ for $n=9$), and finally a second decay phase. 

Although the structured ansatz adopted here \cite{Kocher2025} (see Figure~\ref{fig:ansatz_n5}) is tailored to nonlinear Schrödinger-type problems, the resulting global optimization over the full parameter space becomes increasingly demanding as the number of qubits grows. 
As the spatial resolution increases, the number of qubits required to represent the discretized problem grows, thereby enlarging the Hilbert-space dimension. As discussed in \cite{McClean2018,Cerezo2021}, sufficiently expressive circuits may exhibit exponentially vanishing gradient variance as the Hilbert-space dimension grows. They can more generally be associated with exponentially small loss differences \cite{Larocca2024ReviewBP}. From this perspective, our ``relative energy change'' between successive training steps provides an empirical proxy for the onset of a BP. This three-stage behavior --- initial decay, extended plateau, and subsequent decay --- to the best of our knowledge, does not appear to have been explicitly highlighted in the literature.

This observation motivates the development of localized optimization strategies that reduce the effective optimization dimension while preserving the physical structure of the underlying PDE. In Section~\ref{sec:dd_strategy}, we introduce a domain-decomposition approach designed to achieve this goal.

\section{Domain Decomposition Strategy}
\label{sec:dd_strategy}
The full-domain formulation \eqref{eq:global-vqa}
optimizes a single parameter vector $\bm{\theta}\in\mathbb{R}^p$
to represent the entire discretized wavefunction.
In this global setting, all degrees of freedom contribute
simultaneously to the energy functional,
resulting in strong global coupling among parameters. As illustrated in the previous section, such global expressivity promotes gradient concentration
and BP-like flatness as the system size increases. To mitigate this scalability bottleneck, as suggested in \cite{Cerezo2021} where the authors propose to localize the cost function,
we introduce a domain decomposition framework
that partitions the spatial domain into localized subdomains.
Rather than optimizing a single global parameter vector,
we perform sequential updates on smaller interacting parameter blocks,
thereby reducing long-range parameter correlations
and modifying the effective optimization geometry.

%The distinction between classical and variational quantum domain decomposition is essential. 
Even though the basic concept is the same—dividing the initial domain into overlapping subdomains—the ``variational quantum'' implementation of the domain decomposition algorithm is fundamentally different from the ``classical'' implementation. In classical methods, the wavefunction is explicitly represented in physical space, and subdomain updates directly modify pointwise solution values.
These updates are retained, so each iteration builds cumulatively upon these previously computed physical values. In contrast, within a VQA framework,
the wavefunction is implicitly defined by a parameterized quantum circuit.
Subdomain updates modify circuit parameters rather than physical degrees of freedom.
Consequently, one cannot resume a subdomain computation from stored pointwise values;
only parameter configurations can be transferred between iterations.
This structural constraint fundamentally changes how domain decomposition operates in the quantum setting. In addition, the effectiveness of quantum domain decomposition relies on reducing the Hilbert space dimension and parameter complexity of each subproblem.
If the global discretization uses $N=2^n$ grid points (requiring $n$ qubits),
each subdomain should contain fewer points. Since the number of points (represented by qubits) has to be a power of $2$, this means that it will use, at most $2^{n-1}$ points, based on $n-1$ qubits.
%of size $N/2$ requires only $n-1$ qubits.
As discussed earlier, since gradient variance deteriorates with increasing Hilbert space dimension, reducing the number of qubits improves trainability. Partitioning the domain into only two overlapping subdomains, however,
is not possible. This implies that we must introduce, at least, three overlapping subdomains that allows to further localize the optimization
and improve separation between parameter blocks. %This motivates the three subdomains construction described below.

\subsection{Three-Subdomain Decomposition and Local Quantum Ansatz}

We consider a uniform discretization of the periodic interval
$\Omega=[0,2\pi)$ with
$N = 2^n$
grid points and mesh size $\Delta x = 2\pi/N$.
The global discrete wavefunction
\begin{equation}\label{eq:global_wavefucntion}
\bm{\psi}=(\psi_0,\dots,\psi_{N-1})^\top \in \mathbb{C}^N    
\end{equation}
is represented using $n$ qubits and satisfies the discrete normalization constraint
\begin{equation}
\Delta x \sum_{j=0}^{N-1} |\psi_j|^2 = 1.
\label{eq:global_norm}
\end{equation}
To localize the optimization, we partition the index set
$\{0,\dots,N-1\}$ into $K=3$ overlapping subsets
\[
\mathcal{I}_1,\ \mathcal{I}_2,\ \mathcal{I}_3
\subset
\{0,\dots,N-1\},
\]
such that
\begin{equation}
\bigcup_{k=1}^{3} \mathcal{I}_k
=
\{0,\dots,N-1\}.
\label{eq:subdomain_cover}
\end{equation}
 and each subset contains at most -- and in our case we choose exactly -- $
|\mathcal{I}_k| = 2^{n-1} = \frac{N}{2}$ consecutive indices (modulo $N$)\footnote{
Point indices are considered modulo 
$N$, so that index $0$ comes immediately after index $N-1$; similarly subdomain indices are considered modulo $3$.}.
Consequently, each local subproblem, with degrees of freedom being the  values $\psi_j, j\in \mathcal{I}_k$ for $k=1, 2, 3$   can be represented using only $(n-1)$ qubits,
reducing the Hilbert space dimension from $2^n$ to $2^{n-1}$.
As discussed in Section~\ref{sec:bp},
since gradient variance in variational quantum circuits deteriorates with increasing Hilbert space dimension,
this qubit reduction improves trainability.

Because
\[
3 \cdot 2^{n-1} > 2^n,
\]
the subdomains necessarily overlap.
The total overlap cardinality is therefore
\[
3 \cdot 2^{n-1} - 2^n
=
2^{n-1}.
\]

Assuming pairwise overlaps only, we define
\[
|\mathcal{I}_1 \cap \mathcal{I}_2| = n_1, \quad
|\mathcal{I}_2 \cap \mathcal{I}_3| = n_2, \quad
|\mathcal{I}_3 \cap \mathcal{I}_1| = n_3,
\]
with
\[
n_1 + n_2 + n_3 = 2^{n-1}.
\]

To distribute the overlap as uniformly as possible while preserving integer cardinalities for all $n$, we define
\begin{equation}
n_1 = n_2 = \frac{2^{n-1}+(-1)^n}{3},
\qquad
n_3 = \frac{2^{n-1}+(-1)^n}{3} + (-1)^{\,n-1}.
\end{equation}

This construction defines three overlapping subdomains on the periodic torus.
Figure~\ref{fig:three_subdomain_pointwise} illustrates the decomposition in index space,
while Figure~\ref{fig:three_subdomain_circle} shows the corresponding geometric representation on $\Omega=[0,2\pi)$ for $5$ qubits\footnote{Note that our choice of having $|\mathcal{I}_k| = 2^{n-1}$ for $k=1, 2, 3$ leads to a subdomain definition that depends slightly on $n$ with no consequence on the algorithm. If we want to avoid this dependance we can of course rather choose to have fewer points on each fixed size subdomains.}.

\begin{figure}[h!]
\centering
\includegraphics[
width=0.6\textwidth,
trim=0cm 0.8cm 0cm 0cm,
clip]{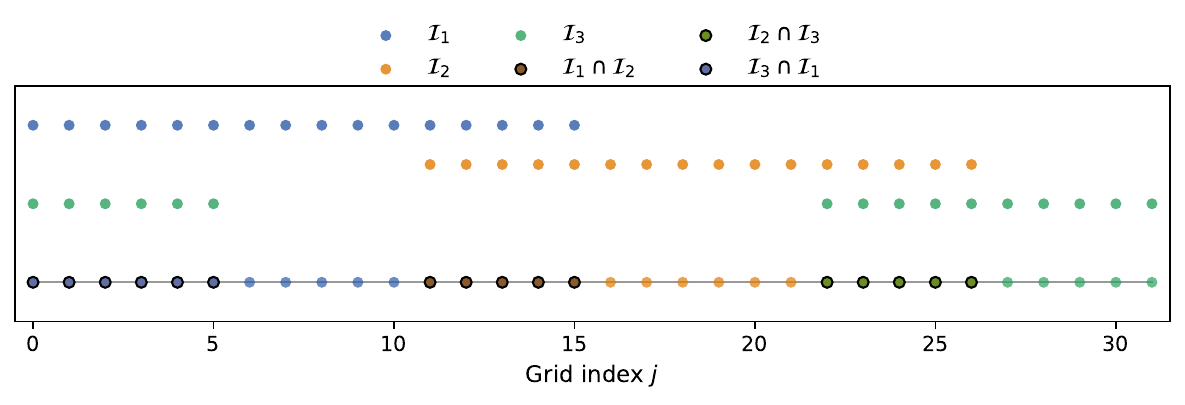}
\caption{Three overlapping subdomains in index space.
The rows show the indices belonging to the subdomains
$\mathcal{I}_1$, $\mathcal{I}_2$, and $\mathcal{I}_3$.
The bottom row displays the full grid index set,
with colored markers highlighting the overlap regions.}
\label{fig:three_subdomain_pointwise}
\end{figure}

\begin{figure}[h!]
\centering
\includegraphics[
width=0.6\textwidth,
trim=0cm 1.7cm 0cm 0.3cm,
clip
]{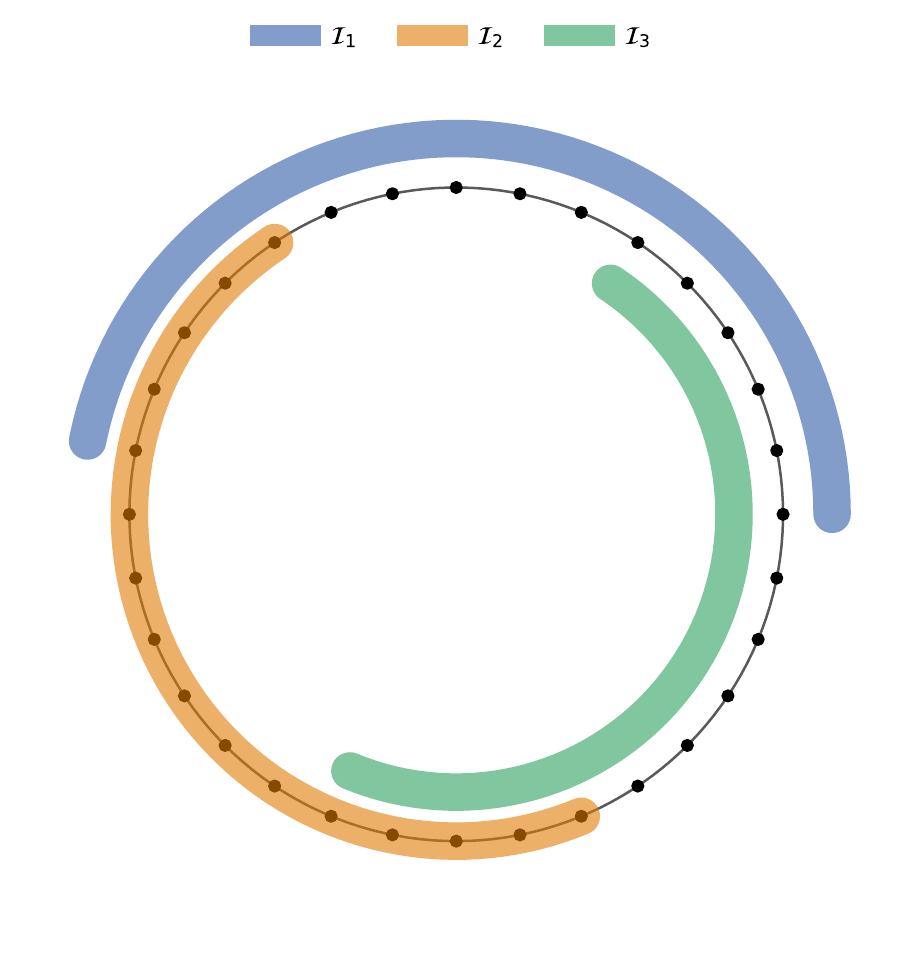}

\caption{Geometric illustration of three overlapping subdomains on the periodic domain
$\Omega=[0,2\pi)$. Black points represent the discrete grid nodes, while the arcs
indicate the three index subsets $\mathcal{I}_1$, $\mathcal{I}_2$, and $\mathcal{I}_3$.}

\label{fig:three_subdomain_circle}
\end{figure}

For each subdomain indexed by $k$, we introduce a local parameter vector
\[
\bm{\theta}^{(k)} \in [0, 2\pi)^{p_{\mathrm{loc}}},
\qquad
p_{\mathrm{loc}}\in \mathbb{N}
\]
together with a localized quantum ansatz
\begin{equation}
\bm{\phi}^{(k)}(\bm{\theta}^{(k)})
\in \mathbb{C}^{|\mathcal{I}_k|},
\end{equation}
obtained from a state-vector-level simulation of a quantum circuit.
By construction,
\[
\|\bm{\phi}^{(k)}(\bm{\theta}^{(k)})\|_2 = 1.
\]

Note however, that in the discretized PDE setting, it is the global wave-function $\bm{\psi} \in \mathbb{C}^N$ that must satisfy the discrete normalization condition \eqref{eq:global_norm}. We now describe a simple strategy to construct a globally normalized update from a localized quantum ansatz $\bm{\phi}^{(k)}$.

Assume we are given a global wave-function  $\bm{\psi}^{\rm old}\in\mathbb{C}^N$
that satisfies the global normalization constraint \eqref{eq:global_norm}, i.e.,
\begin{align*}
  \Delta x\Vert   \bm{\psi}^{\rm old}\Vert_{2} =\Delta x\sum_{\ell=0}^{N-1} \vert \bm{\psi}^{\rm old}_{\ell}\vert^2= 1.
\end{align*}

Let $\mathcal{I}_k \subset \{0,\dots,N-1\}$ be a contiguous set of grid indices corresponding to subdomain $k$. 
Denote its leftmost and rightmost indices by
\[
b_L := \min \mathcal{I}_k,
\qquad
b_R := \max \mathcal{I}_k,
\]
and define
\[
\partial \mathcal{I}_k := \{ b_L, b_R \}, 
\qquad
\mathcal{I}_k^\circ := \mathcal{I}_k \setminus \partial \mathcal{I}_k,
\qquad
\mathcal{I}_k^c := \{0,\dots,N-1\} \setminus \mathcal{I}_k.
\]
Assume now that we have computed -- using, e.g., some optimization procedure -- a localized quantum state
\begin{align*}
    \bm{\phi}^{(k)} = \{{\phi}_\ell^{(k)}\}_{\ell \in \mathcal{I}_k} \in\mathbb{C}^{\vert \mathcal{I}_k\vert} 
\end{align*}
that satisfies the local normalisation constraint
\begin{align*}
     \Vert   \bm{\phi}^{(k)} \Vert_{2} =\sum_{\ell\in \mathcal{I}_k}\vert {\phi}^{(k)}_{\ell}\vert^2= 1,
\end{align*}
and we wish to use the localized quantum state $\bm{\phi}^{(k)}$ to update the global wave-function $\bm{\psi}^{\rm old}$. Notice that we cannot simply define the updated global wave-function $\bm{\psi}^{\rm new}\in \mathbb{C}^{N}$ as taking values $\bm{\psi}^{\rm old}_{\ell}$ on indices $\ell \in \mathcal{I}_k^{\rm c}$ and taking values $\bm{\phi}^{(k)}_{\ell}$ on indices $ \ell \in \mathcal{I}_k$ since the resulting wave-function does not satisfy the global normalization constraint \eqref{eq:global_norm}. It is useful to express the above construction via a so-called ``embedding operator''  $\mathcal{E}_k :
\mathbb{C}^N\times \mathbb{C}^{|\mathcal{I}_k|} 
\longrightarrow
\mathbb{C}^N$ by setting
\begin{align}\label{eq:embedding}
   \mathcal{E}_k(\bm{\psi}^{\rm old}, \bm{\phi}^{(k)})  :=\begin{cases}
    \alpha(\bm{\psi}^{\rm old}, \bm{\phi}^{(k)})\bm{\phi}^{(k)}_{\ell} \quad &\text{for } \ell \in \mathcal{I}^{\circ}_k\\
        \bm{\psi}^{\rm old}_\ell \quad &\text{for } \ell \in \mathcal{I}_k^{\rm c} \cup \partial \mathcal{I}_{k}.
    \end{cases}
\end{align}
where the scaling factor $\alpha(\bm{\psi}^{\rm old}, \bm{\phi}^{(k)}) \in \mathbb{R}$, that depends on  $ \bm{\psi}^{\rm old}\in \mathbb{C}^N$ and $\bm{\phi}^{(k)}\in \mathbb{C}^{\vert \mathcal{I}_k\vert}$ is defined by:
\begin{align}\label{eq:scaling_factor}
    \alpha(\bm{\psi}^{\rm old}, \bm{\phi}^{(k)})=  \sqrt{\Delta x}\sqrt{\frac{\sum_{\ell \in\mathcal{I}_k^{\circ}} |\bm{\psi}^{\rm old}_\ell|^2}{
\sum_{\ell \in \mathcal{I}_k^{\circ}}
\left|\bm{\phi}^{(k)}_{\ell}
\right|^2}}
\end{align}
With this notation, we thus propose to define $
    \bm{\psi}^{\rm new} = \mathcal{E}_k(\bm{\psi}^{\rm old}, \bm{\phi}^{(k)}) .$
Note that this definition (through \eqref{eq:embedding})  ensures consistency between the initial global wave-function $\bm{\psi}^{\rm old}\in \mathbb{C}^{N}$ and the update $\bm{\psi}^{\rm new} \in \mathbb{C}^{N}$ at the `boundary' grid-points corresponding to indices in $\partial\mathcal{I}_k$. Moreover, $\bm{\psi}^{\rm new}$ satisfies the global normalization constraint \eqref{eq:EXN} by construction.

% To do so, it will be useful to introduce the vector
% \begin{equation}
% \bm{\psi}^{(p)}(\bm{\theta}^{(p)})
% =
% \frac{\bm{\phi}^{(p)}(\bm{\theta}^{(p)})}{\sqrt{\Delta x}}.
% \label{eq:psi_raw}
% \end{equation}
% Note that $\tilde{\bm{\psi}}^{(p)}$ is the localized analogue of the global discrete wave-function $\bm{\psi} \in \mathbb{C}^N$ (recall Equation \eqref{eq:global_wavefucntion}) and it satisfies
% \begin{align}\label{eq:local_norm}
%   \Vert   \bm{\psi}^{(p)}(\bm{\theta}^{(p)}) \Vert_{2} = \frac{1}{\Delta x}\sum_{\ell=0}^{\vert \mathcal{I}_p \vert -1} \vert \bm{\phi}^{(p)}_{\ell}(\bm{\theta}^{(p)})\vert^2= \frac{1}{\Delta x}.
% \end{align}

% in order to define the new global solution $\widetilde{\bm{\phi}}^{(p)}$ $\tilde{\bm{\psi}}^{(p)}$ cannot directly be combined with the If this raw subdomain were inserted directly into the global state
% while keeping the remaining indices fixed,
% the discrete normalization constraint \eqref{eq:global_norm}
% would generally be violated,
% since only a subset of degrees of freedom is modified.
% A consistent redistribution of $L^2$ mass is therefore required,
% which motivates the embedding construction described next.

\subsection{Sequential Subdomain Optimization}

In the sequel, the localized quantum state $\bm{\phi}^{(k)}= \bm{\phi}^{(k)}(\bm{\theta}^{(k)})$ will often be parameterised by some $\bm{\theta}^{(k)} \in [0, 2\pi)^{p_{\rm loc}}$ that is obtained via minimising a local energy functional on each subdomain. For notational convenience therefore, we shall frequently use the short-hand notation 
\begin{align*}
    \bm{\psi}^{(k)}(\bm{\theta}^{(k)}; \bm{\psi}):= \mathcal{E}_k(\bm{\psi}, \bm{\phi}^{(k)}(\bm{\theta}^{(k)})).
\end{align*}
The associated local energy functional on subdomain $k$ is given by
\begin{equation}
\mathcal{C}_k(\bm{\theta}^{(k)};\bm{\psi})
=
E_n\!\left(
\bm{\psi}^{(k)}(\bm{\theta}^{(k)};\bm{\psi})
\right),
\label{eq:subdomain_cost_siam}
\end{equation}
and the local optimization problem reads
\begin{equation}
\bm{\theta}^{(k),\star}
=
\arg\min_{\bm{\theta}^{(k)}\in\mathbb{R}^{p_{\mathrm{loc}}}}
\mathcal{C}_k(\bm{\theta}^{(k)};\bm{\psi}).
\label{eq:subdomain_argmin}
\end{equation}

% {\color{blue} Hassan: Would it make sense to move the above embedding construction to the previous subsection and rename this section to domain decomposition algorithm or similar?}\laila{I moved the definition of the embedding operator $\mathcal{E}_k$ and the associated normalization scaling to the previous subsection, since this part defines the construction used to insert a localized quantum state into the global wave-function.
% I kept the definition of the local state $\bm{\psi}^{(k)}(\bm{\theta}^{(k)};\bm{\psi})$, the cost functional $\mathcal{C}_k(\bm{\theta}^{(k)};\bm{\psi})$, and the corresponding minimization problem in the subsequent subsection, as these elements belong to the optimization step of the algorithm rather than to the embedding construction.}

Equipped with all of the ingredients necessary, we now define our proposed VQA domain decomposition strategy. Prior to presenting numerical results, let us briefly describe this algorithm. 

The global problem is solved through solving a sequence of local problems. A sweep, indexed by $s=0,1,2,\dots$, consists of one complete cycle in which all subdomains
$k=1,\dots,3$ are updated sequentially. Let $\bm{\psi}^{0}\in\mathbb{C}^N$ be an initial globally normalized state.
In the numerical experiments we initialize with the uniform constant state
\[
\psi^{0}_j = \frac{1}{\sqrt{2\pi}},
\qquad j=0,\dots,N-1.
\]
For a given sweep $s$, define the intra-sweep iterates
\[
\bm{\psi}^{s,0} := \bm{\psi}^{s},
\]
then, for each subdomain $k=1,\dots,3$, compute an approximate solution of
\begin{equation}
\bm{\theta}^{(k),\star}
\approx
\arg\min_{\bm{\theta}^{(k)}\in\mathbb{R}^{k_{\mathrm{loc}}}}
\mathcal{C}_k\bigl(\bm{\theta}^{(k)};\bm{\psi}^{s,k-1}\bigr),
\label{eq:block_update}
\end{equation}
and update the global state via
\begin{equation}
\bm{\psi}^{s,k}
=
\mathcal{E}_k\!\left(
\bm{\psi}^{s,k-1},
\tilde{\bm{\psi}}^{(k)}(\bm{\theta}^{(k),\star})
\right).
\label{eq:block_embedding}
\end{equation}
After all $K=3$ subdomains have been updated, we define
\[
\bm{\psi}^{s+1} := \bm{\psi}^{s,K},
\]
and proceed to the next sweep. Each subdomain update preserves the global normalization by construction,
since the embedding operator enforces the exact mass allocation
 at every step. Consequently, all intermediate iterates $\bm{\psi}^{s,k}$ remain feasible,
and the sequential subdomain updates generate a normalized sequence
$\{\bm{\psi}^s\}_{s\ge 0}$.
The full domain-decomposition procedure described above is summarized in Algorithm~\ref{alg:dd_embedding}.
\begin{algorithm}[H]
\caption{Sequential subdomain optimization with norm-preserving embedding }
\label{alg:dd_embedding}
\begin{algorithmic}[1]

\Require
Number of subdomains $K$, index sets $\{\mathcal{I}_k\}_{k=1}^{K}$,
global energy functional $E_n(\cdot)$,
initial normalized state $\bm{\psi}^0\in\mathbb{C}^N$,
local ansatz maps
$\bm{\theta}^{(k)}\mapsto\bm{\phi}^{(k)}(\bm{\theta}^{(k)})\in\mathbb{C}^{|\mathcal{I}_k|}$.

\Ensure
Sequence of normalized iterates $\{\bm{\psi}^s\}_{s\ge0}$.

\For{$s=0,1,2,\dots,S-1$} \Comment{sweeps}

\State $\bm{\psi}^{s,0}\gets\bm{\psi}^{s}$

\For{$k=1,\dots,K$} \Comment{subdomain updates}

\State Define the local cost
\[
\mathcal{C}_k(\bm{\theta}^{(k)};\bm{\psi}^{s,k-1})
=
E_n\!\left(
\mathcal{E}_k\!\left(
\bm{\psi}^{s,k-1},
\bm{\phi}^{(k)}(\bm{\theta}^{(k)})
\right)
\right)
\]

\State Compute
\[
\bm{\theta}^{(k),\star}
\approx
\arg\min_{\bm{\theta}^{(k)}}
\mathcal{C}_k(\bm{\theta}^{(k)};\bm{\psi}^{s,k-1})
\]

\State Update
\[
\bm{\psi}^{s,k}
\gets
\mathcal{E}_k\!\left(
\bm{\psi}^{s,k-1},
\bm{\phi}^{(k)}(\bm{\theta}^{(k),\star})
\right)
\]

\EndFor

\State $\bm{\psi}^{s+1}\gets\bm{\psi}^{s,K}$

\EndFor

\end{algorithmic}
\end{algorithm}

\subsection{Comparison with Classical Domain Decomposition} \label{subsec:dd_classical_vs_vqa} In classical domain decomposition methods, the wavefunction is explicitly represented in physical space. At sweep $s$, when updating subdomain $k$, the algorithm directly optimizes the physical values $\{\psi_j : j \in \mathcal I_k\}$ while keeping the remaining components fixed. More precisely, using the intermediate notation established previously, the classical subdomain-$k$ update solves 
\begin{equation} 
\bm{\psi}^{s,k} = \underset{\bm{\psi}\in \mathcal{A}_k(\bm{\psi}^{s,k-1})}{\text{\rm argmin}}E_n(\bm{\psi}), 
\end{equation}
over a feasible set defined by the fixed boundary node condition, i.e., 
\[ \mathcal{A}_k(\bm{\psi}^{s,k-1}) := \left\{ \bm{\psi}\in\mathbb{C}^N : \psi_j = \psi^{s,k-1}_j \ \text{for all } j \notin \mathcal I_k^\circ \right\}. \] 
Since the full physical state $\bm{\psi}^{s,k-1}$ is known explicitly, its restriction to the subsequent subdomain, $\bm{\psi}^{s,k-1}\big|_{\mathcal I_{k}}$, is directly accessible. Therefore, when updating subdomain $k$, the optimization starts from the exact physical values of the current iterate. Consequently, the starting state always belongs to the feasible set, guaranteeing that, for any $s$
\[ E_n(\bm{\psi}^{s,k}) \le E_n(\bm{\psi}^{s,k-1}). \] 
Hence, under local minimization, the global energy is non-increasing during classical updates. 

In contrast, within the VQA framework, the subdomain update is not performed directly on the physical values. Instead, each subdomain is parameterized by a quantum circuit, and the update is restricted to the expressible image of the ansatz: 
\[\mathcal{A}_k^{\mathrm{VQA}}(\bm{\psi}^{s,k-1}):=
\left\{
\mathcal{E}_k\!\left(
\bm{\psi}^{s,k-1},
\bm{\phi}^{(k)}(\bm{\theta}^{(k)})
\right)
:
\bm{\theta}^{(k)}\in\mathbb{R}^{p_{\mathrm{loc}}}
\right\}.\]

When the algorithm begins to optimize each subdomain $k$, a fundamental challenge arises: although the reference physical values $\bm{\psi}^{s,k-1}\big|_{\mathcal I_{k}}$ are known, there is no explicit inverse mapping to deduce the parameter vector $\bm{\theta}^{(k)}$ that would exactly reproduce this state. Consequently, the local VQA optimization for subdomain $k$ cannot start from a parameter that exactly represents the current global iterate $\bm{\psi}^{s,k-1}$. Instead, a warm-start strategy is employed, typically reusing the parameters $\bm{\theta}^{(k)}$ from the previous sweep ($s-1$). As a result, at the beginning of the optimization step for each subdomain $k$, the global wavefunction temporarily deviates from the previously improved state, leading to temporary upward spikes in the total energy before the local optimizer reduces it again.
\begin{figure}[h] \centering \includegraphics[width=\textwidth]{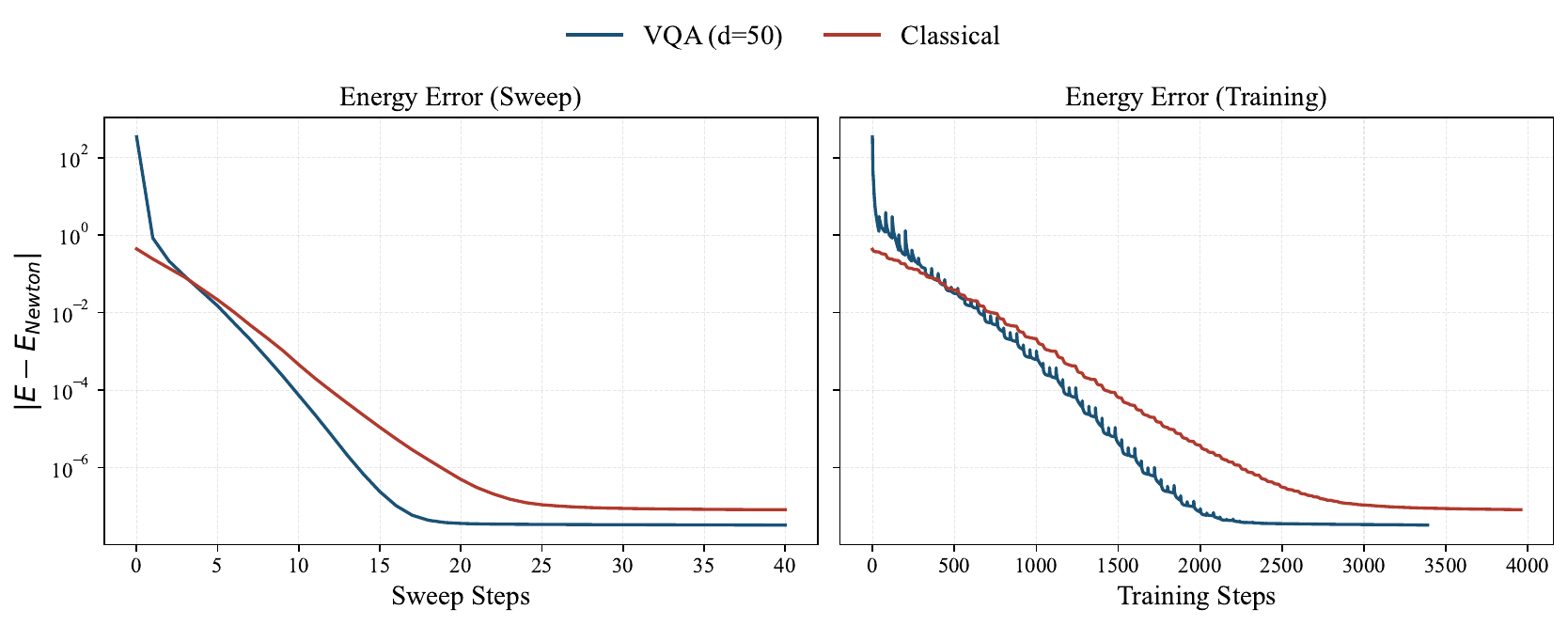} \caption{Comparison of training dynamics between the classical domain decomposition and the VQA domain decomposition ($d=50$) for n=$7$.} \label{fig:classical_vs_vqa} \end{figure} 

Figure~\ref{fig:classical_vs_vqa} illustrates the training dynamics, and, in particular, the temporary increases in the energy error observed in the VQA-based domain decomposition. In the global energy-error curve (measured against total training steps), the VQA-based domain decomposition exhibits these increases in the energy at the beginning of subdomain optimization. However, global energy-error curve, measured at the end of each sweeps resumes a monotonic decrease. This indicates that the observed spikes do not reflect an intrinsic instability of the VQA formulation, but rather arise from the sequential subdomain-update mechanism together with the absence of an exact inverse map from physical states to circuit parameters.

In principle, this effect could be alleviated in future implementations by introducing  more effective parameter-synchronization strategies, thereby reducing the loss of representational consistency across subdomains. By contrast, the classical domain decomposition scheme guarantees a monotonic decrease of the energy at each subdomain update, although it may stagnate at higher energy levels. The VQA-based formulation, on the other hand, exhibits a faster decay of the energy error when measured with respect to training steps.

We stress, however, that the classical method considered here is implemented using the same sequential subdomain-update protocol as the VQA approach in order to ensure a controlled comparison. More advanced classical domain decomposition strategies could display different convergence properties. Accordingly, the present comparison should be understood as a study under identical update dynamics, designed to isolate the effect of the variational ansatz, rather than as a definitive comparison between classical and quantum domain decomposition methods. It illustrates also that the above warm-start strategy makes sense.

%Numerical Results: Mitigation of BP-like Flatness
\section{Numerical Results}
\label{sec:numerical_results}
\subsection{Experimental Setup}
\label{subsec:setup}
We compare the global (full-domain) VQA formulation with the proposed
domain-decomposed VQA.
Unless stated otherwise, all experiments use the potential
\eqref{potential},
and the same hardware-efficient ansatz architecture. We consider $n=7, 8$, and $9$ qubits (thus $N=2^n$ grid points) with circuit depths of $d=100, 200$, and $400$, respectively, in the full-domain formulation. For the domain-decomposed VQA, each subdomain contains one fewer qubit.
Thus, a global problem with $n$ qubits is decomposed into subproblems
with $n-1$ qubits.
Accordingly, the circuit depth used for each subdomain is reduced by a
factor of two (i.e., $d=50, 100$, and $200$ for $n-1=6, 7$, and $8$, respectively). Specifically, in all domain-decomposition experiments, the number of optimization steps per subdomain update is fixed to $50$. To ensure a fair comparison across system sizes, all VQA runs are initialized from the same constant parameter vector,
i.e., $\theta_i = 1$ for all circuit parameters.
As a reference, we compute a high-accuracy ground state solution using a classical Newton method and denote its energy by $E_{\rm Newton}$. We report (i) the absolute $L_2$ error with respect to the Newton reference, and (ii) the energy error $|E - E_{\rm Newton}|$. All local and global optimizations are performed using the BFGS algorithm with a stopping tolerance of $10^{-20}$.

\subsection{Controlled Subdomain Training to Avoid Prolonged Flat Regimes}

The domain-decomposition optimization is implemented as a sequential
subdomain update procedure: during sweep $s$, each subdomain $k$ is updated by
solving a local constrained minimization problem while the exterior
degrees of freedom remain fixed. We first consider the regime of \emph{maximal local training}, where
each local subproblem is solved to numerical convergence. More precisely,
for each sweep $s$ and each subdomain $k$, the local problem
\eqref{eq:subdomain_argmin} is minimized using BFGS until machine precision
is reached (tolerance $10^{-20}$). This process is repeated for a fixed
number of five sweeps.
\begin{figure}[h!]
\centering
\includegraphics[width=\textwidth]{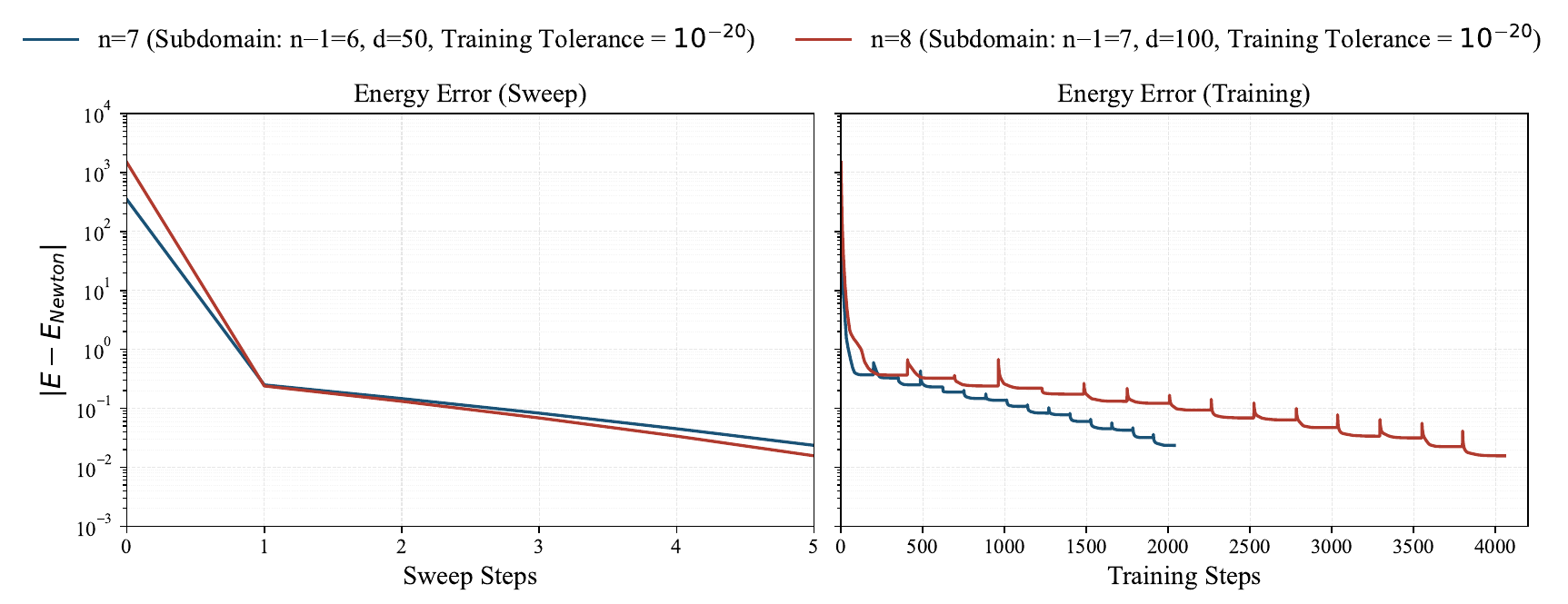}
\caption{
Domain decomposition dynamics under maximal subdomain training
(machine-precision convergence) for $n=7,8$.
Left: energy error per sweep.
Right: energy error versus total training steps.
}
\label{fig:dd_max_training}
\end{figure}
Figure~\ref{fig:dd_max_training} illustrates the training dynamics
for $n=7$ and $n=8$ under this maximal-training regime.
Although domain decomposition mitigates part of the BP
behavior, it does not eliminate it entirely. From the global energy-error curve (measured against total training
steps), we observe that achieving a comparable energy error level requires
approximately twice as many local iterations per subdomain in the $n=8$
case compared to $n=7$, despite both configurations performing the
same number of sweeps. Interestingly, the energy decrease within each subdomain update is nearly identical across both cases. This indicates that solving each local problem to full numerical convergence may prolong the overall optimization by allowing the algorithm to spend excessive iterations in locally flat regions.

We therefore consider a \emph{controlled-training regime}, where the
number of optimization steps for each subdomain is fixed a priori
(e.g., $50$ BFGS iterations per subdomain), regardless of whether local
numerical convergence has been reached. In principle, the number of optimization steps required for each subdomain may increase as the number of qubits grows, since the associated optimization landscape becomes more complex. However, to ensure a fair comparison across different system sizes,
we fix the number of optimization steps per subdomain to $50$ for all
values of $n$ in our experiments.

\begin{figure}[h!]
\centering
\includegraphics[width=\textwidth]{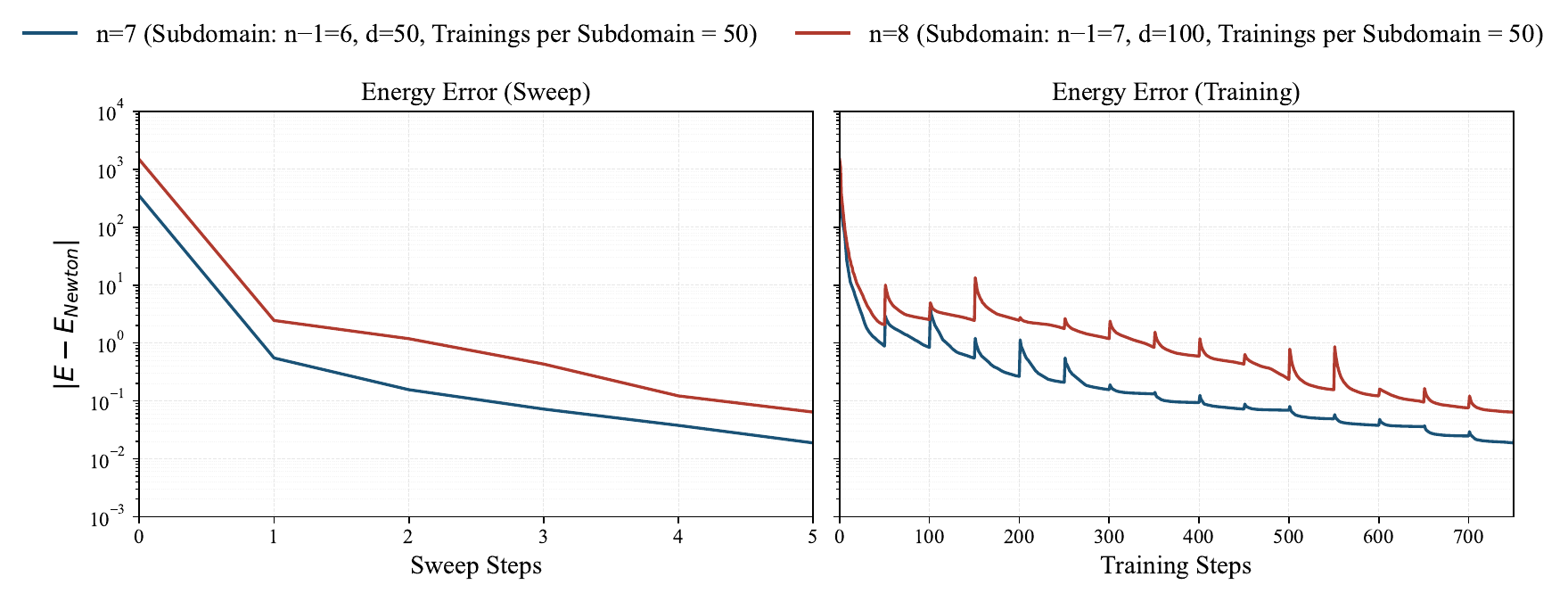}
\caption{
Domain decomposition dynamics with fixed subdomain training
(50 iterations per subdomain update) for $n=7,8$.
Left: energy error per sweep.
Right: energy error versus total training steps.
}
\label{fig:dd_50_training}
\end{figure}

Figure~\ref{fig:dd_50_training} shows the corresponding dynamics.
Under the controlled schedule, the discrepancy between $n=7$ and
$n=8$ is substantially reduced. In contrast to the maximal-training
case, the $n=8$ configuration no longer remains trapped in extended
flat regions. As a result, both cases achieve comparable energy errors
under a similar overall iteration budget. This observation suggests that limiting the number of optimization steps per subdomain
can prevent excessive exploration of locally flat regions and thereby
mitigate prolonged BP-like behavior.

\subsection{Comparison Between Full-Domain and Domain-Decomposed VQA}
\label{subsec:comparison}

Figure~\ref{fig:dd_vs_full} highlights a qualitative difference between the two formulations.
In the full-domain setting, increasing $n$ leads to visibly slower reduction of the energy error and a flatter training trajectory,
consistent with the BP-like energy landscape flatness discussed in Section~\ref{sec:bp}.
In contrast, when domain decomposition is applied, the flattening trend is substantially reduced:
the $L_2$ error and the energy error decrease more uniformly across system sizes, indicating improved scalability of the optimization.
\begin{figure}[h]
\centering
\includegraphics[width=\textwidth]{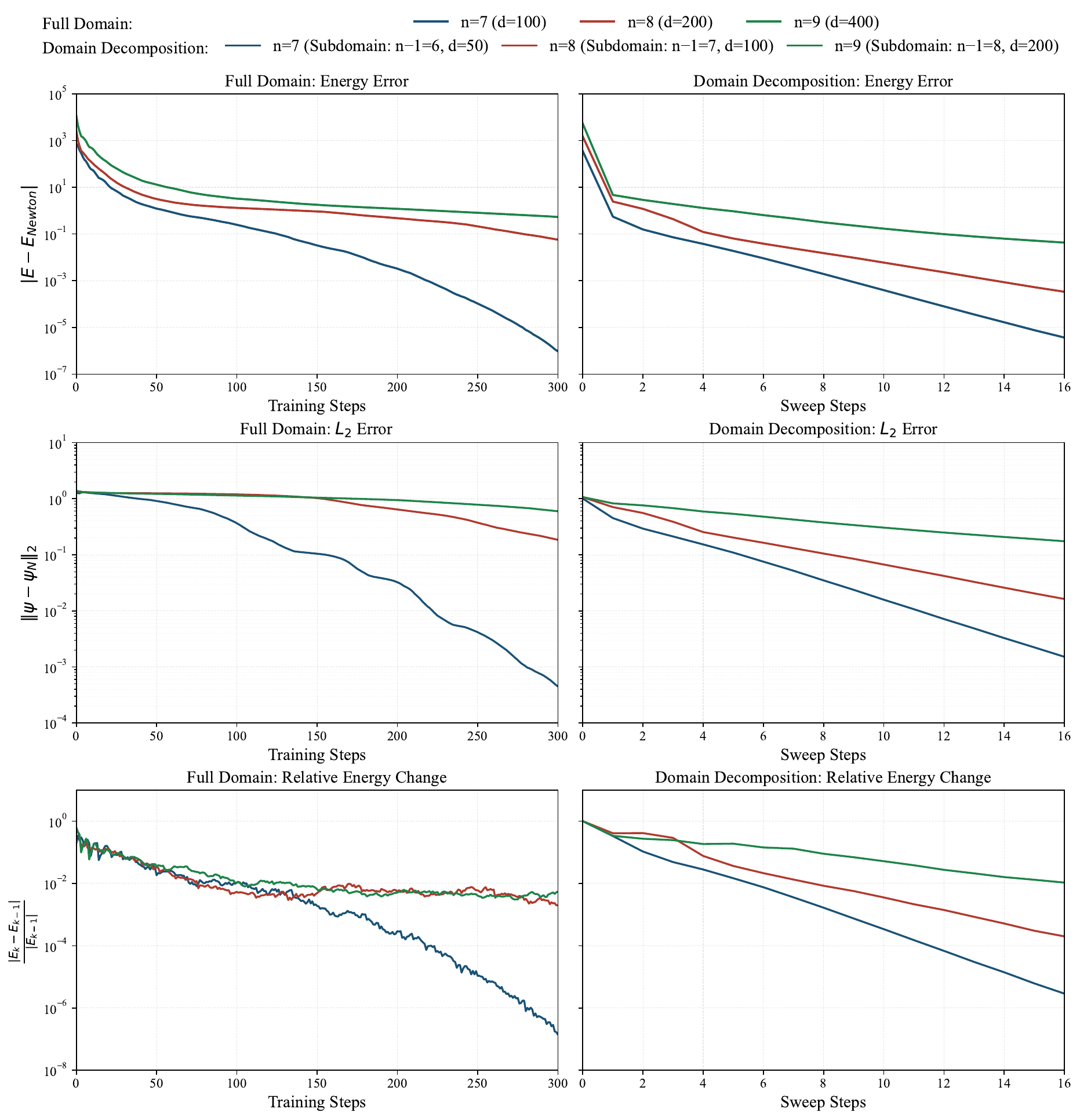}
\caption{
Comparison between the full-domain VQA and the domain-decomposed VQA for n=$7,8,9$. Left column: full-domain; right column: domain decomposition. Top row: energy error.
Middle row: $L_2$ error. Bottom row: relative energy change.
}
\label{fig:dd_vs_full}
\end{figure}
A direct quantitative comparison between the two formulations is not straightforward, since the computational cost of a single BFGS iteration depends on several implementation factors, including the evaluation of the kinetic term and the behavior of the line-search procedure. To obtain a rough estimate of the relative computational cost per BFGS iteration, we consider several simple proxies.

First, the parameter dimension differs significantly.
In the full-domain formulation, the ansatz contains $2n(d+1)$ parameters.
In contrast, in the domain-decomposition formulation each subdomain ansatz contains $2(n-1)(d/2+1)$ parameters.
For the configurations considered in Section~\ref{subsec:setup},
the full-domain ansatz therefore involves roughly $2.2$--$2.3$ times more parameters than each subdomain ansatzs.

Second, each BFGS iteration requires evaluating the gradient of the energy with respect to all parameters.
The full-domain formulation evaluates the energy on $2^n$ grid nodes, whereas each subdomain involves only $2^{n-1}$ nodes.
This suggests an additional factor of roughly two in the cost of evaluating each gradient component.

Finally, the BFGS optimization involves line-search procedures and quasi-Newton updates whose computational overhead increases with the dimension of the parameter space.
Although the precise impact of these effects depends on problem-specific features of the optimization landscape and cannot be predicted reliably a priori, the quasi-Newton update itself operates on matrices whose dimension is proportional to the number of optimization parameters.
Since the full-domain ansatz involves roughly twice as many parameters as each subdomain ansatz, the associated update operations can therefore be expected to introduce an additional factor of approximately two in the optimization overhead.
Combining these effects leads to a rough estimate that a complete BFGS iteration in the full-domain formulation may be on the order of $8$--$10$ times more expensive than in the subdomain case.

Based on this heuristic computational cost estimate, we compare the full-domain and domain-decomposition formulations in terms of the $L_2$ error and the energy error $|E-E_{\rm Newton}|$ under a comparable computational budget.
Using the estimate that a full-domain iteration is roughly eight times more expensive than a subdomain update, we establish an approximately equivalent iteration budget between the two formulations.
Under this assumption, $300$ BFGS iterations in the full-domain formulation correspond to approximately $16$ sweeps of the domain decomposition method when each subdomain is updated with $50$ training steps.

\begin{table}[h!]
\centering
\caption{Comparison of the absolute $L_2$ error and energy error under an approximately equivalent computational cost ($300$ full-domain trainings vs $16$ domain decomposition sweeps)}
\label{tab:quantitative_budget}
\begin{tabular}{lcccccc}
\toprule
\textbf{Method} 
& \multicolumn{2}{c}{$n=7$} 
& \multicolumn{2}{c}{$n=8$} 
& \multicolumn{2}{c}{$n=9$} \\
\cmidrule(lr){2-3} \cmidrule(lr){4-5} \cmidrule(lr){6-7}
& $|E-E_{\rm Newton}|$ & $L_2$ Error & $|E-E_{\rm Newton}|$ & $L_2$ Error & $|E-E_{\rm Newton}|$ & $L_2$ Error \\
\midrule
Full domain 
& $9.5\times 10^{-7}$ & $4.5\times 10^{-4}$ 
& $5.7\times 10^{-2}$ & $1.8\times 10^{-1}$ 
& $6.0\times 10^{-1}$ & $5.3\times 10^{-1}$ \\

Domain decomposition 
& $3.7\times 10^{-6}$ & $1.5\times 10^{-3}$ 
& $3.4\times 10^{-4}$ & $1.6\times 10^{-2}$ 
& $4.3\times 10^{-2}$ & $1.7\times 10^{-1}$ \\
\bottomrule
\end{tabular}
\end{table}

Table~\ref{tab:quantitative_budget} indicates that the domain decomposition formulation mitigates the BP phenomenon compared to the full-domain formulation.
For $n=7$, the full-domain formulation achieves smaller $L_2$ and energy errors than the domain decomposition method.
In this relatively small regime, the BP effect is not yet dominant, allowing the full-domain formulation to benefit from the deeper circuit ansatz ($d=100$ vs.\ $d=50$).
However, as the number of qubits increases, the convergence of the full-domain formulation rapidly deteriorates.
The domain decomposition method also exhibits some degradation in convergence as the number of qubits increases; however, this deterioration is significantly less pronounced than in the full-domain formulation.
Overall, these results suggest that domain decomposition provides a practical strategy for stabilizing VQA optimization as the problem size grows.

\subsection{Effect of Circuit Depth on the Optimization Dynamics}
\label{subsec:depth_effect}
In the previous experiments, the circuit depth in each subdomain was chosen to be half of that used in the corresponding full-domain formulation. Here, we further investigate the effect of the circuit depth on the optimization dynamics within the domain-decomposition formulation. Specifically, we compare two settings: one where each subdomain uses the same circuit depth as the full-domain ansatz, and another where the depth is reduced by a factor of two. This corresponds to comparing depths $d=50$ and $100$ for $n=7$ (subdomain: $n-1$=$6$), $d=100$ and $200$ for $n=8$ (subdomain $n-1$=$7$), and $d=200$ and $400$ for $n=9$ (subdomain $n-1$=$8$). In all experiments, the number of optimization steps per subdomain update is fixed to $50$.

\begin{figure}[h]
\centering
\includegraphics[width=\textwidth]{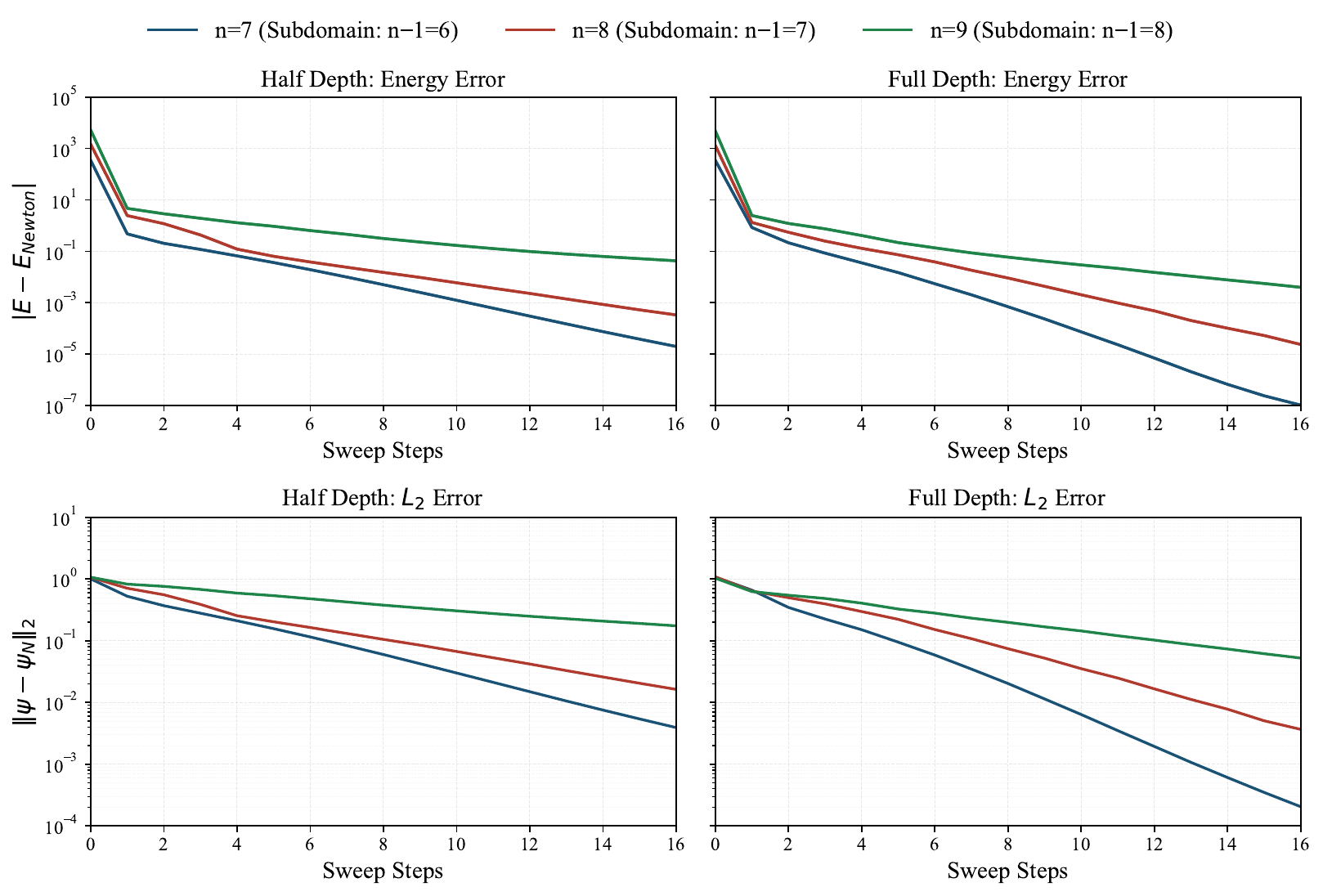}
\caption{
Comparison between half-depth and full-depth in domain decomposition formulation for n=$7,8,9$. Left column: half-Depth; right column: full-Depth. Top row: energy error. Bottom row: $L_2$ error.
}
\label{fig:DD_comparison}
\end{figure}

\begin{table}[h!]
\centering
\caption{Effect of circuit depth on the optimization performance in the
domain-decomposition formulation after 16 sweeps.}
\label{tab:circuit_depth}
\begin{tabular}{lcccccc}
\toprule
\textbf{Circuit Depth} 
& \multicolumn{2}{c}{$n=7$} 
& \multicolumn{2}{c}{$n=8$} 
& \multicolumn{2}{c}{$n=9$} \\
\cmidrule(lr){2-3} \cmidrule(lr){4-5} \cmidrule(lr){6-7}
& $|E-E_{\rm Newton}|$ & $L_2$ Error & $|E-E_{\rm Newton}|$ & $L_2$ Error & $|E-E_{\rm Newton}|$ & $L_2$ Error \\
\midrule
Half depth 
& $3.7\times 10^{-6}$ & $1.5\times 10^{-3}$ 
& $3.4\times 10^{-4}$ & $1.6\times 10^{-2}$ 
& $4.3\times 10^{-2}$ & $1.7\times 10^{-1}$ \\

Full depth
& $1.0\times 10^{-7}$ & $2.1\times 10^{-4}$ 
& $2.4\times 10^{-5}$ & $3.7\times 10^{-3}$ 
& $4.0\times 10^{-3}$ & $5.2\times 10^{-2}$ \\
\bottomrule
\end{tabular}
\end{table}

Figure~\ref{fig:DD_comparison} 
%and %Table~\ref{tab:depth_and_sweeps_convergence}
first show the clear trade-off between computational cost and ansatz expressivity. When the circuit depth is reduced by half, the optimization converges to solutions with larger energy and $L_2$ errors for all tested $n$. These results suggest that, although reducing the circuit depth decreases the computational cost per subdomain update, a sufficiently expressive ansatz remains important for achieving accurate solutions in the domain-decomposition framework.

\section{Conclusion and Perspective}

In this work, we introduced a domain-decomposition framework for variational quantum
approximations of the time-independent Gross-Pitaevskii eigenvalue problem. The proposed
strategy partitions the global domain $\Omega$ into overlapping subdomains,
each represented using a reduced number of qubits. A norm-preserving
embedding operator was constructed to ensure exact conservation of the
discrete $L_2$ norm at every intermediate subdomain update, so that all
iterates generated by the subdomain optimization procedure remain
feasible with respect to the global normalization constraint. The reduction of each local subproblem from $n$ qubits to $(n-1)$ qubits
decreases the effective Hilbert space dimension and modifies the
optimization landscape. Numerical experiments that show that the global
formulation exhibits clear barren plateau-like degradation as the system size
increases, manifested by flattened energy error trajectories and reduced
progress under fixed training budgets. In contrast, the proposed
domain-decomposition strategy alleviates this effect and improves scaling in both
$L_2$ error and energy error. Within the domain decomposition framework, we further examined the role of the local training schedule within the
sequential block updates. Solving each local subproblem to full numerical
convergence may lead to long stagnation in the energy profile and thus, unnecessary expenditure of computational resources.  
On the other hand, limiting the number of optimization steps per subdomain update
improves the overall energy decay while maintaining a reasonable computational cost. Future work can, e.g., focus on balancing the number of local optimization steps on each subdomain and the number of global domain decomposition iterations.

\section*{Acknowledgements} 
This project has received funding from the European Research Council (ERC) under the European Union’s Horizon 2020 research and innovation program (Grant Agreement No. 810367). M. H. acknowledges support from the Deutsche Forschungsgemeinschaft (DFG) through the Emmy Noether Programme (Project No. 555300205). L. S. B. acknowledges support from Imam Abdulrahman Bin Faisal University and 
King Abdullah University of Science and Technology (KAUST). Y.M. also acknowledge the support of the project HQI initia-
tive (www.hqi.fr) supported by France 2030 under
the French National Research Agency award number “ANR-
22-PNCQ-0002"/

The authors would like to thank the organizers of CEMRACS 2025 for making this collaboration possible and for providing the opportunity to meet and work together.

\bibliographystyle{plain}
\bibliography{refs}

\vfill\eject

\appendix

\section*{\textbf{SUPPLEMENTAL INFORMATION}}
\label{appendix}
\section{Lie-algebraic perspective on barren plateaus and domain decomposition}

In this appendix, we interpret the training behavior observed in the main text from the viewpoint of dynamical Lie algebras. The purpose is twofold: first, to show that the full-domain ansatz generates the whole Lie algebra $\mathfrak{su}(2^n)$ and is therefore maximally expressive; second, to explain why the domain-decomposition strategy effectively restricts the variational search to lower-dimensional subproblems, thereby improving trainability.

\subsection{Dynamical Lie algebras and variance scaling}

Let $U(\bm\theta)$ be a parametrized quantum circuit acting on an $n$-qubit Hilbert space $\mathcal H_n=(\mathbb C^2)^{\otimes n}$, and let
\[
\mathfrak g_n \subseteq \mathfrak{su}(2^n)
\]
denote the dynamical Lie subalgebra (DLA) generated by this circuit. For a fixed input state $|\psi\rangle \in (\mathbb C^2)^{\otimes n}$ and an observable $O$, the associated cost function is defined by
\begin{equation}
\mathcal C(\bm\theta)
=
\langle \psi |\, U(\bm\theta)^\dagger O\, U(\bm\theta)\, |\psi\rangle .
\end{equation}

Recent Lie-algebraic analyses of BP \cite{Larocca_2022,Ragone2024} show that, under suitable assumptions, the variance of $\mathcal C(\bm\theta)$ satisfies
\begin{equation}
\label{eq:appendix_variance}
\mathrm{Var}[\mathcal C(\bm\theta)]
=
\frac{\mathcal P_{\mathfrak g_n}(\psi)\,\mathcal P_{\mathfrak g_n}(O)}{\dim(\mathfrak g_n)},
\end{equation}
where $\mathcal P_{\mathfrak g_n}(\psi)$ and $\mathcal P_{\mathfrak g_n}(O)$ quantify the overlap of the input state and the observable with the dynamical Lie subalgebra. This expression shows that, as the dimension of the Lie algebra increases, the cost function becomes increasingly concentrated. In particular, if $\dim(\mathfrak g_n)$ scales exponentially with the number of qubits, then the fluctuations of the cost function are expected to become exponentially small, which is one of the standard signatures of BP.

The quantities $\mathcal P_{\mathfrak g_n}(\psi)$ and $\mathcal P_{\mathfrak g_n}(O)$ may affect the prefactor in \eqref{eq:appendix_variance}, but they are not expected, in general, to compensate for the decay induced by an exponential growth of $\dim(\mathfrak g_n)$.

Strictly speaking, \eqref{eq:appendix_variance} applies in the deep-circuit regime in which the ansatz forms a unitary 2-design over the Lie group generated by the DLA. Since this property is not established here for the finite circuit depths used in the simulations, \eqref{eq:appendix_variance} should be understood as an asymptotic heuristic providing a qualitative Lie-algebraic interpretation of the observed training behavior.

The central question is therefore to determine the dimension of the DLA generated by the variational ansatz.
\subsection{Definition of the full-domain DLA}

Let
 $\X_j,\Y_j,\Z_j$ denote the Pauli matrices acting on qubit $j\in\{1,\dots,n\}$. We denote by
\[
\mathcal P_n
=
\left\{
\sigma_1\otimes\cdots\otimes\sigma_n
\;:\;
\sigma_j\in\{\I,\X,\Y,\Z\}
\right\}
\]
the set of $n$-qubit Pauli strings, and by
\[
\mathcal P_n^\star = \mathcal P_n \setminus \{I^{\otimes n}\}
\]
the set of nontrivial Pauli strings.

The Lie algebra of the special unitary group $SU(2^n)$ is
\begin{equation}
\label{eq:appendix_su}
\mathfrak{su}(2^n)
=
\left\{
A\in\mathbb C^{2^n\times 2^n}
\;:\;
A^\dagger=-A,
\quad
\mathrm{Tr}(A)=0
\right\},
\end{equation}
and satisfies
\[
\dim\bigl(\mathfrak{su}(2^n)\bigr)=4^n-1.
\]

The ansatz considered in the main text consists of repeated layers of local rotations $R_x$ and $R_z$ on every qubit, followed by a fixed entangling block $U_{\mathrm{ent}}$ composed of nearest-neighbour CNOT gates on a connected graph. The corresponding infinitesimal generators are therefore
\begin{equation}
\label{eq:appendix_local_generators}
\{\, i\X_j,\ i\Z_j : j=1,\dots,n \,\}.
\end{equation}
Since the same entangling block is repeated from layer to layer, the DLA generated by the ansatz is
\begin{equation}
\label{eq:appendix_dla}
\mathfrak g_n
=
\Big\langle
i\X_j,\ i\Z_j,\ iU_{\mathrm{ent}}\X_jU_{\mathrm{ent}}^\dagger,\ iU_{\mathrm{ent}}\Z_jU_{\mathrm{ent}}^\dagger
\;:\;
j=1,\dots,n
\Big\rangle_{\mathrm{Lie}}.
\end{equation}

We now show that $\mathfrak g_n$ coincides with the full Lie algebra $\mathfrak{su}(2^n)$.

\subsection{Auxiliary commutator identities}

We shall use the standard Pauli commutation relations
\begin{equation}
\label{eq:appendix_pauli_comm}
[\X,\Y]=2i\Z,
\qquad
[\Y,\Z]=2i\X,
\qquad
[\Z,\X]=2i\Y,
\end{equation}
as well as the fact that conjugation by a CNOT gate maps Pauli operators to Pauli operators. More precisely, if $\mathrm{CX}_{a\to b}$ denotes a CNOT with control qubit $a$ and target qubit $b$, then
\begin{equation}
\label{eq:appendix_cnot_conj}
\mathrm{CX}_{a\to b}\,\X_a\,\mathrm{CX}_{a\to b}^\dagger = \X_a\X_b,
\qquad
\mathrm{CX}_{a\to b}\,\Z_b\,\mathrm{CX}_{a\to b}^\dagger = \Z_a\Z_b.
\end{equation}
Thus, the entangling block transforms local generators into nonlocal Pauli operators supported on neighbouring qubits.

\subsection{Full controllability of the global ansatz}

\begin{theorem}
\label{thm:full_dla}
Let $\mathfrak g_n$ be the dynamical Lie algebra generated by the full-domain ansatz defined in \eqref{eq:appendix_dla}. Assume that the entangling block $U_{\mathrm{ent}}$ consists of nearest-neighbour CNOT gates on a connected coupling graph. Then
\[
\mathfrak g_n = \mathfrak{su}(2^n).
\]
In particular,
\[
\dim(\mathfrak g_n)=4^n-1.
\]
\end{theorem}

\begin{proof}
We prove that $\mathfrak g_n$ contains all traceless Pauli strings.

\smallskip
\noindent
\emph{Step 1: local controllability on each qubit.}
By definition, $i\X_j,i\Z_j\in\mathfrak g_n$ for all $j$. From \eqref{eq:appendix_pauli_comm},
\[
[i\Z_j,i\X_j]\propto i\Y_j,
\]
hence
\[
i\X_j,\ i\Y_j,\ i\Z_j \in \mathfrak g_n
\qquad
\text{for all } j=1,\dots,n.
\]

\smallskip
\noindent
\emph{Step 2: generation of two-qubit Pauli operators on each edge.}
Since $U_{\mathrm{ent}}$ contains CNOT gates along every edge of a connected graph, \eqref{eq:appendix_cnot_conj} implies that $\mathfrak g_n$ contains nonlocal two-qubit terms supported on neighbouring qubits. For any edge $(j,k)$ of the coupling graph, one obtains, up to relabelling,
\[
i\X_j\X_k \in \mathfrak g_n
\qquad \text{or} \qquad
i\Z_j\Z_k \in \mathfrak g_n.
\]
Assume that $i\X_j\X_k\in\mathfrak g_n$ for some neighbouring pair $(j,k)$. Since $i\Z_j,i\Z_k\in\mathfrak g_n$ and $\mathfrak g_n$ is closed under commutators, it follows that
\[
[i\Z_j,i\X_j\X_k]\propto i\Y_j\X_k,
\qquad
[i\Z_k,i\X_j\X_k]\propto i\X_j\Y_k.
\]
Hence $i\Y_j\X_k,i\X_j\Y_k\in\mathfrak g_n$, and further commutators generate all two-qubit Pauli strings supported on the edge $(j,k)$.
 Therefore,
\[
i\sigma_j^\alpha \sigma_k^\beta \in \mathfrak g_n,
\qquad
\forall (\alpha,\beta)\in\{\X,\Y,\Z\}^2,
\]
for every neighbouring pair $(j,k)$.

\smallskip
\noindent
\emph{Step 3: propagation to connected supports.}
We now prove by induction on the size of the support that every Pauli string supported on a connected subset of qubits belongs to $\mathfrak g_n$.

The result is true for support size one by Step 1 and for support size two by Step 2. Assume it holds for all connected supports of size $r$. Let $S$ be a connected subset with $|S|=r$, and let $k\notin S$ be a neighbouring qubit such that $S\cup\{k\}$ remains connected. By the induction hypothesis, $\mathfrak g_n$ contains every Pauli string supported on $S$, and by Step 2 it also contains all two-qubit Pauli operators on the edge linking $k$ to $S$. Choosing a two-qubit Pauli operator that does not commute with the factor acting on the boundary qubit of $S$, one obtains by commutation a new Pauli string supported on $S\cup\{k\}$. Hence the claim follows by induction.

\smallskip
\noindent
\emph{Step 4: generation of arbitrary Pauli strings.}
Since the coupling graph is connected, every nontrivial Pauli string can be built by successively extending its support along a spanning tree and adjusting local Pauli labels by means of Step 1. Therefore,
\[
iP\in\mathfrak g_n
\qquad
\forall P\in\mathcal P_n^\ast.
\]
As the set $\{\,iP : P\in\mathcal P_n^\ast\,\}$ forms a basis of $\mathfrak{su}(2^n)$, we conclude that
\[
\mathfrak g_n=\mathfrak{su}(2^n).
\]
The dimension formula follows immediately:
\[
\dim(\mathfrak g_n)=4^n-1.
\]
\end{proof}

Theorem \ref{thm:full_dla} shows that the global ansatz is fully controllable \cite{Larocca_2022}. In particular, its DLA has exponential dimension, which places the full-domain optimization in the Lie-algebraic regime where BP phenomena are expected.

\subsection{Local Lie algebras in the domain-decomposition setting}

We now consider the domain-decomposition strategy. At each local optimization step, the circuit does not act on the full $n$-qubit register, but only on a subdomain represented on
\[
n_{\mathrm{sub}} = n-1
\]
qubits. The same hardware-efficient ansatz is then applied on this reduced register.

Let $\mathfrak g_{n-1}$ denote the DLA generated on one such subdomain. Since the local circuit has the same structure as the global ansatz, Theorem \ref{thm:full_dla} immediately yields the following statement.

\begin{proposition}
\label{prop:subdomain_dla}
The dynamical Lie algebra associated with one domain-decomposition subproblem is
\[
\mathfrak g_{n-1}=\mathfrak{su}(2^{n-1}).
\]
In particular,
\[
\dim(\mathfrak g_{n-1})=4^{n-1}-1.
\]
\end{proposition}

\begin{proof}
The proof is identical to that of Theorem \ref{thm:full_dla}, with $n$ replaced by $n-1$.
\end{proof}

Therefore, each local optimization step in the domain-decomposition scheme explores a Lie algebra whose dimension is four times smaller, up to lower-order corrections, than that of the full problem:
\[
\frac{4^n-1}{4^{n-1}-1}\approx 4.
\]

\subsection{Interpretation of the domain-decomposition dynamics}

The two key points are
thus :
\begin{itemize}
    \item the domain-decomposition algorithm does not optimize over $\mathfrak{su}(2^n)$ in a single variational step. Instead, it performs a sequence of local optimizations, each of which is restricted to a subproblem of dimension
\[
4^{n-1}-1
\]
at the Lie-algebraic level. Thus, although each subdomain ansatz remains fully expressive on its own reduced register, the global procedure avoids exploring the full $n$-qubit unitary manifold at once.
\item we do not perform a full minimization of the subdomain cost functional, because the solution outside the subdomain currently being optimized is not yet converged. As a result, fully converging the local problem would be unnecessary at this stage; it is enough to obtain a sufficient improvement of the state on the active subdomain.
\end{itemize}

This observation provides a natural explanation for the improved trainability seen in the numerical experiments. The full-domain ansatz is maximally expressive, with
\[
\mathfrak g_n=\mathfrak{su}(2^n),
\]
while domain decomposition preserves local expressivity while reducing the effective Lie-algebraic complexity encountered at each optimization stage.

\end{document}